\documentclass[11pt]{amsart}
\usepackage{hyperref} 

\usepackage{amsthm}
\usepackage{amsmath,amssymb}
\usepackage{amsxtra}
\usepackage[latin1]{inputenc}
\usepackage[T1]{fontenc}
\usepackage{amssymb,latexsym,array}
\usepackage[UKenglish]{babel}
\usepackage[all]{xy}

\newcommand{\C}{\mathbb C}

\newcommand{\R}{\mathbb R}
\newcommand{\Q}{\mathbb Q}

\newcommand{\F}{\mathbb F}

\newcommand{\rationals}{\mathbb Q}
\newcommand{\Hy}{\mathcal{H}}
\newcommand{\ringO}{\mathcal{O}}

\newcommand{\mat}{\begin{pmatrix}a & b\\c & d\end{pmatrix}}

\renewcommand{\leq}{\leqslant}
\renewcommand{\geq}{\geqslant}

\renewcommand{\O}{{\mathcal O}}

\newcommand{\gl}{{\rm GL}}
\renewcommand{\sl}{{\rm SL}}

\renewcommand{\dim}{{\rm dim}}

\newtheorem{theorem}{\bfseries Theorem}
\newtheorem{proposition}[theorem]{\bfseries Proposition}
\newtheorem{corollary}[theorem]{\bfseries Corollary}
\newtheorem{lemma}[theorem]{\bfseries Lemma}

\theoremstyle{definition}
\newtheorem{notation}[theorem]{\bfseries Notation}
\newtheorem{definition}[theorem]{\bfseries Definition}
\newtheorem{question}[theorem]{\bfseries Question}
\newtheorem{remark}[theorem]{\bfseries Remark}

\begin{document}

\title{On Level One Cuspidal Bianchi Modular Forms}
\author{Alexander D. Rahm}
\email{Alexander.Rahm@nuigalway.ie}
\urladdr{http://www.maths.nuigalway.ie/\char126rahm/}
\address{National University of Ireland at Galway, Department of Mathematics}
\thanks{The first author is funded by the Irish Research Council.} 
\author{Mehmet Haluk \c{S}eng\"{u}n}
\email{M.H.Sengun@warwick.ac.uk}
\urladdr{http://www.uni-due.de/\char126hm0074/}
\address{Mathematics Institiute, University of Warwick, Coventry, UK}
\thanks{The second author is funded by a Marie Curie Intra-European Fellowship.}
\date{\today}

\begin{abstract}
In this paper, we present the outcome of vast computer calculations, locating several of the very rare instances 
of  level one cuspidal Bianchi modular forms that are not lifts of elliptic modular forms.  
\end{abstract}

\maketitle

Bianchi modular forms over an imaginary quadratic field $K$ are automorphic forms of cohomological type associated to the $\Q$-algebraic group $\textrm{Res}_{K/ \Q}(\sl_2)$. Even though modern studies of Bianchi modular forms go back to the mid 1960's, most of the fundamental problems surrounding their theory are still wide open. In this paper, we report on our remarkably extensive computations that show the paucity of {\em genuine} level one cuspidal Bianchi modular forms.

Let $S_k(1)$ denote the space of level one weight $k+2$ cuspidal Bianchi modular forms over $K=\Q(\sqrt{-d})$. In their recent paper \cite{fgt}, Finis, Grunewald and Tirao computed the dimension of the subspace $L_k(1)$ of $S_k(1)$ which is formed by (twists of) those forms which arise from elliptic cuspidal modular forms via base-change or arise from a quadratic extension of $K$ via automorphic induction (see \cite{fgt} for these notions). In this paper, we investigate 
numerically how much of $S_k(1)$ is exhausted by $L_k(1)$. There have been previous reports, however of limited size, in the 2009 paper 
\cite{calegari-mazur} of Calegari and Mazur (the computations in this paper were carried out by Pollack and Stein) and in the 2010 paper \cite{fgt} of 
Finis, Grunewald and Tirao.  While the computations in \cite{calegari-mazur} were limited to the case $d=2$, the computations in \cite{fgt} 
covered ten imaginary quadratic fields. The precise scope of the computations in \cite{fgt} is given in Table \ref{table:fgt} below.

\begin{table}[h] 
\centering
\begin{tabular}{|c|c|c|c|c|c|c|c|c|c|c|} \hline
$d$&1&2&3&7&11&19&5&6&10&14 \\ \hline
$k \leq$&104&141&116&132&153&60&60&60&60&60 \\ \hline
\multicolumn{11}{c}{}\\ 
\end{tabular}
\caption{Finis-Grunewald-Tirao test range}\label{table:fgt}
\end{table} 

It was observed in \cite{calegari-mazur} that for $2k\leq 96$, one has $L_{2k}(1)=S_{2k}(1)$. The computations of \cite{fgt} extended 
those of \cite{calegari-mazur}. An interesting outcome of the data they collected is that except in two of the 946 spaces they computed, one has 
$L_k(1)=S_k(1)$. The exceptional cases are $(d,k)=(7,12)$ and $(d,k)=(11,10)$. In both cases, there is a two-dimensional complement to $L_k(1)$ inside $S_k(1)$. 

Using a different and more efficient approach, we computed, over more than 800 processor-days, the dimension of 4986 different spaces $S_k(1)$ over 186 different imaginary quadratic fields. The precise scope of our computations is given in Tables \ref{ours1}, \ref{ours1b} and \ref{ours2}, 
where $D$ and $h$ denote the discriminant and the class number of $K$ respectively. In only 29 of these spaces were we able to observe genuine forms. The precise data about these exceptional cases is provided in Table \ref{exceptional}.

\begin{table}[h] 
\centering
\begin{tabular}{|c|c|c|c|c|c|c|c|c|c|c|c|} \hline
$|D|$  &{\bf 3}&{\bf 4}&{\bf 7}&{\bf 8}&{\bf 11}&{\bf 15}&{\bf 19}&{\bf 20}&{\bf 23}&{\bf 24}&{\bf 31}   \\ \hline
$h$ &1&1&1&1&1&2&1&2&3&2&3 \\ \hline
$k \leq$ &219 &216&217&217&217&115&120&100&83&101&74       \\ \hline
\multicolumn{12}{c}{}\\ \hline

$|D|$  &{\bf 35}&{\bf 39}&{\bf 40}&{\bf 43}&{\bf 47}&{\bf 51}&{\bf 52}&{\bf 55}&{\bf 56}&{\bf 59}&{\bf 67}          \\ \hline
$h$ &2&4&2&1&5&2&2&4&4&3&1 \\ \hline
$k \leq$  &86&67 &73&83&52&75&65&45&55&60&58      \\ \hline
\multicolumn{12}{c}{}\\ \hline

$|D|$  &{\bf 68}&{\bf 71}&{\bf 79}&{\bf 83}&{\bf 84}&{\bf 87}&{\bf 88}&{\bf 91}&{\bf 95}&{\bf 103}&{\bf 104}           \\ \hline
$h$ &4&7&5&3&4&6&2&2&8&5&6 \\ \hline
$k \leq$  &53&38&33&41&50&36&45&50&30&30&32           \\ \hline
\multicolumn{12}{c}{}\\ \hline

$|D|$  &{\bf 107}&{\bf 111}&{\bf 115}&{\bf 116}&{\bf 119}&{\bf 120}&{\bf 123}&{\bf 127}&{\bf 131}&{\bf 132}&{\bf 136}    \\ \hline
$h$ &3&8&2&6&10&4&2&5&5&4&4 \\ \hline
$k \leq$  &35&28&40&33&25&38&35&25&32&33&32                     \\ \hline
\multicolumn{12}{c}{}\\ \hline

$|D|$  &{\bf 139}&{\bf 143}&{\bf 148}&{\bf 151}&{\bf 152}&{\bf 155}&{\bf 159}&{\bf 163}&{\bf 164}&{\bf 167}&{\bf 168}      \\ \hline
$h$ &3&10&2&7&6&4&10&1&8&11&4 \\ \hline
$k \leq$  &29&20&31&21&24&26&19&33&24&18&26                       \\ \hline
\multicolumn{12}{c}{}\\ \hline

$|D|$  &{\bf 179}&{\bf 183}&{\bf 184}&{\bf 187}&{\bf 191}&{\bf 195}&{\bf 199}&{\bf 203}&{\bf 211}&{\bf 212}&{\bf 215}   \\ \hline
$h$ &5&8&4&2&13&4&9&4&3&6&14 \\ \hline
$k \leq$  &24&19&25&25&15&27&17&23&21&17&14  \\ \hline
\multicolumn{12}{c}{}\\ \hline

\end{tabular}
\caption{the scope of our computations, part 1(a)}\label{ours1}
\end{table}

\begin{table}[h] 
\centering
\begin{tabular}{|c|c|c|c|c|c|c|c|c|c|c|c|} \hline

$|D|$     &{\bf 219}&{\bf 223}&{\bf 227}&{\bf 228}&{\bf 231}&{\bf 232}&{\bf 235}&{\bf 239}&{\bf 244}&{\bf 247} &{\bf 248} \\ \hline
$h$ &4&7&5&4&12&2&2&15&6&6&8 \\ \hline          
$k \leq$  &20&14&17&18&13&21&23&12&17&13&16 \\ \hline       
\multicolumn{12}{c}{}\\ \hline

$|D|$  &{\bf 251}&{\bf 255}&{\bf 259}&{\bf 260}&{\bf 263}&{\bf 264}&{\bf 267}&{\bf 271}&{\bf 276}&{\bf 280}&{\bf 283}          \\ \hline
$h$ &7&12&4&8&13&8&2&11&8&4&3 \\ \hline
$k \leq$  &15&14&17&14&13&15&21&12&16&16&17                            \\ \hline
\multicolumn{12}{c}{}\\ \hline

$|D|$  &{\bf 287}&{\bf 291}&{\bf 292}&{\bf 295}&{\bf 296}&{\bf 299}&{\bf 303}&{\bf 307}&{\bf 308}&{\bf 311}&{\bf 312}     \\ \hline
$h$ &14&4&4&8&10&8&10&3&8&19&4 \\ \hline
$k \leq$  &12&19&16&13&13&13&11&15&13&11&13                 \\ \hline
\multicolumn{12}{c}{}\\ \hline

$|D|$  &{\bf 319}&{\bf 323}&{\bf 327}&{\bf 328}&{\bf 331}&{\bf 335}&{\bf 339}&{\bf 340}&{\bf 344}&{\bf 347}&{\bf 355}   \\ \hline
$h$ &10&4&12&4&3&18&6&4&10&5&4 \\ \hline
$k \leq$  &11&12&10&13&14&11&15&14&10&12&13                    \\ \hline
\multicolumn{12}{c}{}\\ \hline

$|D|$  &{\bf 356}&{\bf 359}&{\bf 367}&{\bf 371}&{\bf 372}&{\bf 376}&{\bf 379}&{\bf 383}&{\bf 388}&{\bf 391}&{\bf 395}  \\ \hline
$h$ &12&19&9&8&4&8&3&17&4&14&8 \\ \hline
$k \leq$  &11&9&11&10&12&12&13 &8&11&9&10                     \\ \hline
\multicolumn{12}{c}{}\\ \hline

$|D|$  &{\bf 399}&{\bf 403}&{\bf 404}&{\bf 407}&{\bf 408}&{\bf 411}&{\bf 415}&{\bf 419}&{\bf 420}&{\bf 424}&{\bf 427}    \\ \hline
$h$ &16&2&14&16&4&6&10&9&8&6&2 \\ \hline
$k \leq$  &8&12&9&8&10&12&10&12&11&10&13     \\ \hline
\end{tabular}
\caption{the scope of our computations, part 1(b)}\label{ours1b}
\end{table}

\begin{table}[h] 
\centering
\begin{tabular}{|c|c|c|c|c|c|c|c|c|c|} \hline

$|D|$ &{\bf 431}&{\bf 435}&{\bf 436}&{\bf 439}&{\bf 440}&{\bf 443}&{\bf 447}&{\bf 451}&{\bf 452}           \\ \hline
$h$ &21&4&6&15&12&5&14&6&8  \\ \hline
$k \leq$ &8&13&11&9&8&12&10&12&11                                          \\ \hline
\multicolumn{10}{c}{}\\ \hline

$|D|$  &{\bf 455}&{\bf 456}&{\bf 463}&{\bf 467}&{\bf 471}&{\bf 472}&{\bf 479}&{\bf 483}&{\bf 487}                              \\ \hline
$h$ &20&8&7&7&16&6&25&4&7 \\ \hline
$k \leq$   &7&10&8&10&7&11&7&11&8                                                                         \\ \hline
\multicolumn{10}{c}{}\\ \hline

$|D|$  &{\bf 488}&{\bf 491}&{\bf 499}&{\bf 520}&{\bf 532}&{\bf 547}&{\bf 555}&{\bf 560}&{\bf 568}                                      \\ \hline
$h$ &10&9&3&4&4&3&4&4&4 \\ \hline
$k \leq$  &9&10&10&9&10&11&11&7&10                                                                       \\ \hline
\multicolumn{10}{c}{}\\ \hline

$|D|$  &{\bf 571}&{\bf 595}&{\bf 627}&{\bf 643}&{\bf 667}  &{\bf 696} &{\bf 708}&{\bf 715}&{\bf 723}                                              \\ \hline
$h$ &5&4&4&3&4 &12  &4&4&4  \\ \hline
$k \leq$  &11&9&11&9&9&4  &7&7&9                                                                                     \\ \hline
\multicolumn{10}{c}{}\\ \hline

$|D|$  &{\bf 760}&{\bf 763} &{\bf 795} &{\bf 883}&{\bf 907}&{\bf 955}&{\bf 1003}&{\bf 1027}&{\bf 1051}       \\ \hline
$h$ &4&4&4&3&3&4&4&4&5 \\ \hline
$k \leq$  &7&7&7&6&7&6&6&5&5    \\ \hline
\multicolumn{10}{c}{}\\ \hline

$|D|$  &{\bf 1123}&{\bf 1227} &{\bf 1243}&{\bf 1387}&{\bf 1411}&{\bf 1507}&{\bf 1555}&{\bf 1723} & {\bf 1747}    \\ \hline
$h$    &5&4&4&4&4&4&4&5&5 \\ \hline
$k \leq$ &4 &5 &4&4&4&4&4&3&3   \\ \hline 
\multicolumn{10}{c}{}\\ \hline

$|D|$  &{\bf 1867}& &&&&&&&    \\ \hline
$h$    &5&&&&&&&& \\ \hline
$k \leq$ &3& &&&&&&&   \\ \hline

\end{tabular}
\caption{the scope of our computations, part 2}\label{ours2}
\end{table}

\begin{table}
\begin{tabular}{|c|c|c|c|c|c|c|c|c|c|c|} \hline
$|D|$     &{\bf 7}   &{\bf 11} &{\bf 71} &{\bf 87}  &{\bf 91}   &{\bf 155}    &{\bf 199}   &{\bf 223}    &{\bf 231} &{\bf 339}\\ \hline
$k$      &12 &10 &1  &2    &6     &4         &1      &0        &4   &1  \\ \hline
$\dim$ &2   &2   &2  &2    &2     &2        &4       &2        &2  &2  \\ \hline
\multicolumn{11}{c}{}\\ \hline

$|D|$       &{\bf 344}  &{\bf 407}  &{\bf 408}    &{\bf 408}    &{\bf 408}    &{\bf 415}     &{\bf 435}   &{\bf 435}   &{\bf 435} &{\bf 435}  \\ \hline
$k$           &1     &0      &2        &5        &8       &0         &2       &5   &8   &11 \\ \hline
$\dim$      &2     &2      &2        &2        &2       &2         &2       &2   &2  &2 \\ \hline 
\multicolumn{11}{c}{}\\ \hline

$|D|$      &{\bf 455}  &{\bf 483} &{\bf 571}  &{\bf 571}  &{\bf 643}   &{\bf 760} &{\bf 1003} &{\bf 1003 } &{\bf 1051}  & \\ \hline
$k$         &0      &1    &0      &1      &0      &2 &0 &1 &0       &    \\ \hline
$\dim$      &2      &2    &2     &2       &2     &2  &2 &2 &2       &   \\ \hline
\end{tabular}
\caption{the cases where there are {\em genuine} classes}\label{exceptional}
\end{table}

In Section \ref{section: varieties}, we briefly discuss the conjectural connections between the spaces $S_0(1)$ and Abelian varieties defined over $K$ of $\textrm{GL}_2$-type. In Section \ref{section: comments}, we make some speculations in light of the data we collect. In particular, we pose a question which can be seen as an analogue of Maeda's conjecture for Bianchi modular forms. Finally in Section \ref{section: algorithm}, we explain how we carried out our computations. As usual, the starting point of our approach is the so called ``Eichler-Shimura-Harder" isomorphism which allows us to replace $S_k(1)$ with the cohomology of the relevant Bianchi group with special non-trivial coefficients. Then to compute this cohomology space, we use the program \textit{Bianchi.gp} \cite{BianchiGP}, which analyzes the structure of the Bianchi group via its action on hyperbolic 3-space (which is isomorphic to the 
associated symmetric space $\sl_2(\C)/{\rm SU}_2$). We then feed this group-geometric information into an equivariant spectral sequence that gives us an explicit description of the second cohomology of the Bianchi group, with the relevant coefficients.

{\bf Acknowledgments.} We wish to thank John Cannon, \mbox{John Cremona} and Stephen S. Gelbart for useful discussions. 
We are grateful to Dan Yasaki who reworked his Magma program to compute for us the number field generated by the Hecke eigenvalues 
of the genuine weight 2 cuspidal Bianchi modular forms that we found for the field $\Q(\sqrt{-455})$. 
It is a pleasure to acknowledge our debt to Ulrich G\"ortz and Gabor Wiese for allowing us to 
use the computer nodes at the Institute for Experimental Mathematics of the University of Duisburg-Essen and 
the Mathematics Department of the University of Luxembourg respectively. 
Also, we are grateful to the Weizmann Institute of Science for providing the high performance computing clusters on which the database of cell complexes used for our calculations has been generated.
The second author thanks the Algebra and Geometry Group of the Mathematics Department of the University of Barcelona for the post-doctoral fellowship under which he carried out most of his work that went into this paper. Moreover, he thanks the Mathematical Sciences Research Institute of the University of California and the Max Planck Institute for Mathematics for the wonderful hospitality that he received during his stays. Finally, we would like to thank Frank Calegari, Lassina Demb\'el\'e for their comments on the paper and to Aurel Page for computing the Hecke action on the weight (1,1) cohomology for the field $\Q(\sqrt{-199})$.

\section{Background} \label{section: background}
Let $K$ be an imaginary quadratic field with ring of integers $\mathcal{O}$. Let $\Gamma$ be the Bianchi group $\sl_2(\O)$. 
It is a discrete subgroup of the real Lie group $\sl_2(\C)$ and thus acts discontinuously on hyperbolic 3-space. 
Let $Y_{\Gamma}$ be the quotient hyperbolic 3-fold. The Borel-Serre compactification, see \cite[appendix]{Serre}, $X_{\Gamma}$ of $Y_{\Gamma}$ is a compact 3-fold with boundary~$\partial X_{\Gamma}$ whose interior is homeomorphic to $Y_{\Gamma}$. When the discriminant of $K$ is smaller than $-4$, 
$\partial X_\Gamma$ consists of $h_K$ disjoint 2-tori where $h_K$ is the class number of $K$. 

Given $n \geq 0$, let $\C[x,y]_n$ denote the space of homogeneous polynomials of degree $n$ on variables $x,y$ with complex coefficients. 
$\sl_2(\C)$ acts on this space in the obvious way permitted by the two variables. Consider the $\sl_2(\C)$-module 
$$E_n:= \C[x,y]_n \otimes_{\C} \overline{\C[x,y]}_n$$ 
where the overline on the second factor is to indicate that the action on this factor is twisted with complex conjugation.  
When considered as a $\Gamma$-module, $E_n$ gives rise to a locally constant sheaf $\mathcal{E}_n$ on $Y_\Gamma$ whose stalks 
are isomorphic to $E_n$. Consider the long exact sequence
$$ \hdots \rightarrow H^i_{c}(Y_{\Gamma}; \thinspace \mathcal{E}_{n}) \rightarrow H^{i}(X_{\Gamma}; \thinspace \bar{\mathcal{E}}_{n}) \rightarrow 
H^{i}(\partial X_{\Gamma}; \thinspace \bar{\mathcal{E}}_{n}) \rightarrow \hdots $$
where $H^i_c$ denotes the compactly supported cohomology and $\bar{\mathcal{E}}_{n}$ is a certain natural extension of $\mathcal{E}_{n}$ 
to $X_\Gamma$.

The \textit{cuspidal cohomology} $H^i_{cusp}$ is defined as the image of the compactly supported cohomology. The \textit{Eisenstein cohomology} $H^i_{Eis}$ is the complement of the cuspidal cohomology inside $H^i$ and it is isomorphic to the image of the restriction map inside the cohomology of the boundary. The decomposition $H^i = H^i_{cusp}\oplus H^i_{Eis}$ respects the Hecke action which is defined, as usual, via correspondences on $X_{\Gamma}$.

By construction, the embedding $Y_{\Gamma} \hookrightarrow X_{\Gamma}$ is a homotopy invariance. Together with the fact that $Y_{\Gamma}$ is a 
$K(\Gamma,1)$-space, we get the isomorphisms
$$H^i(X_{\Gamma}; \thinspace \bar{\mathcal{E}}_{n}) \simeq H^i(Y_{\Gamma}; \thinspace  \mathcal{E}_{n}) \simeq H^i(\Gamma; \thinspace  E_{n}).$$
Via these isomorphisms, we define the cuspidal and Eisenstein parts of $H^i(\Gamma; \thinspace E_{n})$.

Let $S_{n}(1)$ denote the space of level one cuspidal Bianchi modular forms (over $K$) of weight $n+2$. It is well known that 
$$S_{n}(1) \simeq H^1_{cusp}(Y_\Gamma; \thinspace \mathcal{E}_{n}) \simeq H^2_{cusp}(Y_\Gamma; \thinspace \mathcal{E}_{n})$$
as Hecke modules. Here the first isomorphism was established by Harder and the second follows from duality, see \cite{ash-stevens}.

In \cite{fgt}, a formula for the dimension of the space $L_n(1)$ has been given for all fields $K$ and weights $n$. We will compare the dimension 
of $L_n(1)$, which we obtain via their formula, to the dimension of $S_n(1)$, which we will obtain via our computer programs. 
The following Proposition will allow us to deduce the size of the cuspidal cohomology, and hence of $S_n(1)$, once we have computed the size of the whole cohomology. It is well-known to the specialists, however for the convenience of the reader we include a proof of it.
 
\begin{proposition} \label{infinity} Let $K$ be an imaginary quadratic field. Then in the above notation
$$ \textrm{dim}~H^2_{Eis}(X_\Gamma; \thinspace \bar{\mathcal{E}}_{n}) = \begin{cases} h_K-1, \ \ \ {\rm if} \ \ n=0, \\ h_K, \ \  \ \ \ \ \ \  else. \end{cases}$$
where $h_K$ is the class number of $K$.
\end{proposition}

\begin{proof} It is well-known (see Theorem 2.1 of \cite{harder-75}) that the map
$$H^2(X_{\Gamma}; \thinspace \bar{\mathcal{E}}_{n}) \longrightarrow H^2(\partial X_{\Gamma}; \thinspace \bar{\mathcal{E}}_{n})$$
is surjective for $n>0$ and its image has codimension one for $n=0$. 

Assume that the discriminant of $K$ is less than $-4$, that is, 
$K$ is not equal to $\Q(i)$ nor $\Q(\sqrt{-3})$. Then the boundary $\partial X_{\Gamma}$ is a disjoint union 
of 2-tori, indexed by the class group of $K$. Below we prove that the dimension of 
$H^2(T; \thinspace \bar{\mathcal{E}}_{n})$ is one for every boundary component $T$ of $\partial X_\Gamma$, which clearly 
gives our claim.

Let $c \in K \cup \{\infty\}$ be a cusp and let $\Gamma_c$ be its stabilizer in $\Gamma$ (which is a parabolic subgroup). Then $\Gamma_c$ is the fundamental group of $T_c$. In fact, $T_c$ is a $K(\Gamma_c,1)$-space. Hence we can turn our attention to computing $H^2(\Gamma_c; \thinspace E_{n})$. 
Composition of the cup product and the well-known perfect pairing $(\cdot, \cdot) : E_{n} \otimes_{\C} E_{n} \rightarrow \C$ (see, for example, Section 2.4. \cite{berger}) gives us a pairing 

$$\xymatrix{ H^0(\Gamma_c; \thinspace  E_{n}) \times H^2(\Gamma_c; \thinspace  E_{n}) \ar[r]^{\ \ \ \ \ \cup} & 
             H^2(\Gamma_c; \thinspace E_{n} \otimes_{\C} E_{n}) \ar[d]^{(\cdot, \cdot)}  \\ 
            &  H^2(\Gamma_c; \thinspace  \C ) \simeq \C.}$$
Here the last isomorphism follows from the fact that $T_c$ is a compact 2-fold (see also proof of Prop.3.5. of \cite{sengun} for 
a direct algebraic argument). Thus the dimension we are looking for is equal to that of $H^0(\Gamma_c; \thinspace  E_{n})$. 
Clearly, if $n=0$, the latter dimension is 1 and thus the dimension of $H^2(\partial X_{\Gamma}; \thinspace \bar{\mathcal{E}}_{n})$ is 
$h_K$ as desired.

Let us now assume that $n \not=0$. Conjugation by a matrix in $\sl(K)$ which takes $c$ to the cusp at infinity induces an isomorphism 
$$\Gamma_c  \simeq \Gamma_\infty = ( \begin{smallmatrix} * & *  \\ 0 & * \\ \end{smallmatrix} ) \subset \sl_2(\mathcal{O}_K).$$
 Consider the normal subgroup 
$\Gamma^+_\infty := ( \begin{smallmatrix} 1 & *  \\ 0 & 1 \\ \end{smallmatrix} )$ of $\Gamma_\infty$. 
 Then $\Gamma^+_\infty$ is a free Abelian group on two generators. We are now going to determine the submodule $E_{n}^{\Gamma^+_\infty}$ of $E_n$  invariant under its action. As the generators are of the form 
$( \begin{smallmatrix} 1 & *  \\ 0 & 1 \\ \end{smallmatrix} )$, it is clear that the vector $x^n \otimes x^n$ is fixed by $\Gamma^+_\infty$.
One shows, proceeding as in Lemma 2.4. of \cite{wiese}, that there are no other fixed vectors. Hence 
$$H^0(\Gamma_\infty^+; \thinspace E_{n}) = E_{n}^{\Gamma^+_\infty}= \langle x^n \otimes x^n \rangle$$is of complex dimension one. Let $\mu := \Gamma_\infty / \Gamma_\infty^+ = \big \{ ( \begin{smallmatrix} \pm 1 & 0 \\ 0 & \pm 1 \\ \end{smallmatrix} )  \big \} $.  As we are considering modules over $\C$, it follows that 
 $$H^0(\Gamma_\infty; \thinspace E_{n}) \simeq H^0(\Gamma_\infty^+; \thinspace E_{n})^{\mu}$$ 
 is the invariant submodule under $\mu$. We easily check that the action of $\mu$ on~$E_{n}$ is trivial, and so 
 $$H^0(\Gamma_c; \thinspace E_{n}) \simeq H^0(\Gamma_\infty^+; \thinspace E_{n})$$ is again of complex dimension one. This completes the proof 
with our assumption of the discriminant of $K$.

When $K$ is $\Q(i)$ or $\Q(\sqrt{-3})$, due to the extra units, the cross-sections of the cusps, which are again parametrized by the class group, are 2-orbifolds whose underlying manifolds are $2$-spheres (torus folded by an involution). As the second cohomology of the $2$-sphere is one dimensional, the result follows.
\end{proof}


\section{Abelian varieties of \textrm{GL}(2)-type} \label{section: varieties}

There is a widely believed conjectural connection between Bianchi newforms of weight 2 over $K$ and Abelian varieties of $\textrm{GL}_2$-type defined over $K$ (see \cite{egm82},\cite{Cremona92},\cite{taylor}) which is expressed in terms of the associated $L$-functions. In particular, an Abelian variety of $\textrm{GL}_2$-type over $K$, that is not definable over $\Q$ nor of $CM$-type, with everywhere good reduction is expected to give rise to newforms in $S_0(1)^+$. Here $S_0(1)^+$ denotes the {\em plus-subspace} of $S_0(1)$ in the sense of \cite{egm82} and \cite{Cremona}. Equivalently, $S_0(1)^+$ can be seen as the space of cuspidal Bianchi modular forms of weight two for $\gl_2(\O_K)$.

In the reverse direction, the newforms in $S_0(1)^+$ are expected\footnote{There are some natural exceptions coming from elliptic newforms with {\it inner twists}, see 
Cremona \cite{Cremona92}, which are avoided if we consider newforms that are not in $L_0(1)$.} to correspond to Abelian varieties of $\textrm{GL}_2$-type over $K$ which have everywhere good reduction. As listed in Table \ref{exceptional}, we have found eight imaginary quadratic fields for which $S_0(1)$ contained non-lifted classes. For only six of these fields, the non-lifted classes were in fact contained in $S_0(1)^+$. In Table \ref{hecke} below, we list the (necessarily totally real) number field $F$ generated by the Hecke eigenvalues of the non-lifted newforms in these six cases.

\begin{table}[h]
\begin{tabular}{|c|c|c|c|c|c|c|} \hline
$|D|$     &{\bf 223} &{\bf 415}  &{\bf 455} &{\bf 571}  &{\bf 643}  &{\bf 1003}  \\ \hline
$F$        &$\Q(\sqrt{2})$ &$\Q(\sqrt{3})$  & $\Q(\sqrt{5})$ & $\Q(\sqrt{5})$ &$\Q$ & $\Q(\sqrt{7})$  \\ \hline
\end{tabular}
\caption{the number field generated by the Hecke eigenvalues of non-lifted newforms in $S_0(1)^+$}\label{hecke}
\end{table}

We have computed these fields using Dan Yasaki's program, see \cite{Yasaki}, in Magma which computes the Hecke action on $S_0(\Gamma_0(\mathfrak{n}))^+$ 
for congruence subgroups of type $\Gamma_0(\mathfrak{n})$ of  Bianchi groups. Note that as this program only treats 
$GL_2$-cohomology with trivial weight, that is $k=0$, we could not have used it for our experiment.
 
Table \ref{hecke} tells us that there should exist an elliptic curve defined over $\Q(\sqrt{-643})$, and not over $\Q$, which has everywhere good reduction and it should be {\em modular}. Indeed we know by Kr\"amer \cite{Kraemer} that there is such an elliptic curve over $\Q(\sqrt{-643})$ and it does seem to be modular, see Scheutzow \cite{scheutzow}. 
Similarly, there should exist Abelian surfaces defined over $\Q(\sqrt{-d})$ with $d=223,415,455,571,1003$ and not over $\Q$, which have everywhere good reduction and real multiplication by $\sqrt{2},\sqrt{3},\sqrt{5},\sqrt{5},\sqrt{7}$ respectively and they should be modular. Locating such surfaces is a highly nontrivial task, see \cite[Section 8]{sengun-survey}.


\section{Comments}  \label{section: comments}

The data collected in this paper make it clear that the spaces of cuspidal Bianchi modular forms of level one are generically made of forms which are not genuine. 
Unfortunately the data are not enough to formulate a quantitative statement about the occurrences of non-lifted forms. Hence the following question remains.

\begin{question} Let $K$ be an imaginary quadratic field. Let $S_k(1)$ denote the 
space of level one cuspidal Bianchi modular forms over $K$ of weight $k+2$. Is it true that there are at most finitely many $k$ for which the space 
$S_k(1)$ contains non-lifted forms ?
\end{question} 

Let us make a comparison with other types of modular forms. 
For the case of Hilbert modular forms (this the case of the algebraic group $\textrm{Res}_{F/ \Q}(\sl_2)$ where $F$ is a totally real field) and 
Siegel modular forms of genus 2 (the case of the algebraic group $\textrm{Sp}_4$), one has considerable amount of genuine level one cuspidal forms. 
However, the case of the modular forms for $\textrm{SL}_3$ is similar to our case. Here in the range of the data collected by Ash and Pollack, see \cite{ash-pollack}, the spaces of level one modular forms for $\textrm{SL}_3$ are completely made of those which are the symmetric square lifts of classical holomorphic modular forms. They in fact conjecture that this is always the case.

It is interesting to note that for the Hilbert and Siegel modular forms of genus 2, where we have plenty of genuine forms, the associated symmetric 
spaces are Hermitian, while for the Bianchi and $\textrm{SL}_3$ modular forms, where there is an extreme paucity of genuine forms, the 
associated symmetric spaces fail to be Hermitian. Is this part of a general phenomenon ?

Next we shall pose a question about the non-lifted newforms in $S_k(1)$ inspired by classical Maeda's conjecture. 
The nontrivial automorphism $\sigma \in \textrm{Gal}(K / \Q)$ of $K$ acts on the set of newforms in $S_k(1) $ as an involution, again denoted $\sigma$. 
Thus for every newform $f$ has a {\em twin}, denoted ${}^{\sigma}f$. The Hecke eigenvalues $c(\cdot,\pi)$ associated to the Hecke 
operators\footnote{Observe that since we are not working within the adelic setting, we only consider Hecke operators which stabilize the connected components 
of the adelic symmetric space.} $T_{\pi}$ satisfy the relation $$c({}^{\sigma}f,\pi)=c(f,\sigma(\pi))$$ 
for every $\pi \in \O$. Recall that just as in the case of elliptic modular forms, for a newform $f$ in $S_k(1)$ with Hecke eigenvalue field $F$, there is a newform $f^{\tau}$ in $S_k(1)$ for every $\tau \in \textrm{Gal}(F / \Q)$ with the property that $c(f^{\tau},\pi)=\tau(c(f,\pi))$ for every $\pi \in \O$. We say that $f$ and the $f^{\tau}$ form one {\em Galois orbit}.
 
\begin{question} \label{maeda} Let $K$ be an imaginary quadratic field. Is it true that for every $k \geq 0$, the set of non-lifted newforms in $S_k(1)$, modulo the action of $\textrm{Gal}(K / \Q)$, forms one Galois orbit ?
\end{question} 

In all except one of the cases where we observed non-lifted newforms, the dimension of the non-lifted subspace was only two. In this case, 
the answer to Question \ref{maeda} is automatically {\em yes} as the two non-lifted newforms have to be twins (that is, the Galois conjugate of the newform is equal to its twin). Aurel Page kindly computed the Hecke action (based on the methods of \cite{page}) on the non-lifted classes for the case $(199,1)$ for us and his data show that the answer to Question \ref{maeda} is {\em yes} in this case as well. More precisely, we have a pair of Galois conjugate non-lifted newforms with coefficients in $\Q(\sqrt{13})$ and their twins, forming the four dimensional non-lifted subspace. 

As Frank Calegari remarked to us, if there are two elliptic curves defined over $K$ with good reduction everywhere and such that neither come from $\Q$ nor are conjugates of each other, then the answer to Question \ref{maeda} would be {\em no} for $S_2(1)$. Note that the analogue of this conjecture for Hilbert modular forms over real quadratic fields holds in the range of the computations performed by Doi and Ishii, see \cite{doi-hida-ishii} p.568.


\section{Method of the computations} \label{section: algorithm}
In this section, we will explain how we computed the cohomology of the investigated Bianchi groups.

Let $m$ be a square-free positive integer and $K = \rationals(\sqrt{-m}\thinspace)$
 be an imaginary quadratic number field with ring of integers $\mathcal{O}_{-m}$, which we also just denote by $\mathcal{O}$. 
Consider the familiar action (we give an explicit formula for it in lemma \ref{operationFormula}) of the group 
\mbox{$\Gamma := \mathrm{SL_2}(\mathcal{O}) \subset \mathrm{GL}_2(\C)$} on hyperbolic three-space,
 for which we will use the upper-half space model $\Hy$. 
 
As a set, $$ \Hy = \{ (z,\zeta) \in \C \times \R \medspace | \medspace \zeta > 0 \}. $$

\begin{lemma}[Poincar\'e]  \label{operationFormula}
If $\gamma = $\scriptsize$ \mat $\normalsize$ \in \mathrm{GL}_2(\C)$, the action of $\gamma$ on $\Hy$ is given by $\gamma \cdot (z,\zeta) = (z',\zeta')$, where
$$ \zeta' = \frac{|\det \gamma|\zeta}{|cz-d|^2 +\zeta^2|c|^2},
\qquad z' = \frac{\left(\thinspace\overline{d -cz}\thinspace\right)(az-b) -\zeta^2\bar{c}a}{|cz-d|^2 +\zeta^2|c|^2}.$$
\end{lemma}

The Bianchi--Humbert theory \cite{Bianchi}, \cite{Humbert} gives a fundamental domain for the action of $\Gamma$ on $\Hy$,
 which we shall call the \emph{Bianchi fundamental polyhedron}.
It is a polyhedron in hyperbolic space up to the missing vertex~$\infty$,
and up to a missing vertex for each non-trivial ideal class if $\ringO_{-m}$ is not a principal ideal domain.
We observe the following notion of strictness of the fundamental domain:
 the interior of the Bianchi fundamental polyhedron contains no two points which are identified by~$\Gamma$.
Swan~\cite{Swan}
 proves a theorem which implies that the boundary of the Bianchi fundamental polyhedron consists of finitely many cells.

\subsection{A cell complex for the Bianchi groups} \label{refined cell complex}
Swan further produces a concept for an algorithm to compute the Bianchi fundamental polyhedron.
Such an algorithm has been implemented by Cremona~\cite{Cremona}
 for the five cases where~$\ringO_{-m}$ is Euclidean, and by his students Whitley~\cite{Whitley}
 for the non-Euclidean principal ideal domain cases, Bygott~\cite{Bygott} for a case of class number 2 and Lingham (\cite{Lingham},
 used in~\cite{CremonaLingham}) for some cases of class number 3; and finally Aran\'es~\cite{Aranes} for arbitrary class numbers.
Another algorithm based on this concept has independently implemented in~\cite{BianchiGP}
 for all Bianchi groups; and we make explicit use of the cell complexes it produces.
Other results of the employed implementation are described in~\cite{RahmThesis}.

\begin{definition}
A pair of elements $(\mu, \lambda) \in \mathcal{O}^2$ is called \emph{unimodular} if the ideal sum $\mu \mathcal{O} +\lambda \mathcal{O}$ equals $\mathcal{O}$.
\end{definition}
The boundary of $\Hy$ is the Riemann sphere $\partial \Hy = \C \cup \{ \infty \}$ (as a set), which contains the complex plane $\C$.
The totally geodesic surfaces in $\Hy$ are the Euclidean vertical planes (we define \emph{vertical} as orthogonal to the complex plane) and the Euclidean hemispheres centered on the complex plane.
\begin{notation} \label{hemispheres}
Given a unimodular pair $(\mu$, $\lambda) \in \mathcal{O}^2$ with $\mu \neq 0$, let $S_{\mu,\lambda} \subset \Hy$ denote the hemisphere given by the equation $|\mu z -\lambda|^2 +|\mu|^2\zeta^2 = 1$.

This hemisphere has centre $\lambda/\mu$ on the complex plane $\C$, and radius $1/|\mu|$.
\label{B}
Let 
\\ $
B:=\bigl\{(z,\zeta) \in \Hy$:
The inequality $|\mu z -\lambda|^2 +|\mu|^2\zeta^2 \geq 1$ 
\begin{flushright} is fulfilled for all unimodular pairs $(\mu$, $\lambda) \in \mathcal{O}^2$ with $\mu \neq 0$ $\bigr\}$.
\end{flushright}
Then $B$ is the set of points in $\Hy$ which lie above or on all hemispheres $S_{\mu,\lambda}$.
\end{notation}

\begin{lemma}[\cite{Swan}] \label{GammaBequalsH}
The set $B$ contains representatives for all the orbits of points under the action of $\mathrm{SL_2}(\mathcal{O})$ on $\Hy$.
\end{lemma}

The action  extends continuously to the boundary \mbox{$\partial \Hy$}, which is a Riemann sphere. \\
In $\Gamma := \mathrm{SL_2}(\mathcal{O}_{-m})$, consider the stabiliser subgroup $\Gamma_\infty$ of the point $\infty \in \partial \Hy$. 
Excluding the two cases $m =1$ and $m=3$ of Gaussian and Eisenstein integers, the latter group  is given as
$$ \Gamma_\infty = \left\{ \pm \begin{pmatrix}1 & \lambda \\ 0 & 1 \end{pmatrix} \thinspace | \medspace \lambda \in \mathcal{O} \right\},$$
which performs translations by the elements of $\ringO$ with respect to the Euclidean geometry of the upper-half space $\Hy$.\\

\begin{notation} \label{fundamentalRectangle}
A fundamental domain for $\Gamma_\infty$ in the complex plane (as a subset of $\partial \Hy$) is given by the rectangle
$$ D_0 := \begin{cases} \{ x +y\sqrt{-m} \in  \C  \medspace | \medspace 0 \leq x \leq 1, \medspace 0 \leq y \leq 1 \}, &
m \equiv 1 \medspace \mathrm{ or } \medspace 2 \mod 4, \\
\{ x +y\sqrt{-m} \in  \C \medspace | \medspace \frac{-1}{2} \leq x \leq \frac{1}{2}, \medspace 0 \leq y \leq \frac{1}{2} \}, &
m \equiv 3 \mod 4.
\end{cases}$$ 

And a fundamental domain for $\Gamma_\infty$ in $\Hy$ is given by
$$ D_\infty := \{ (z, \zeta) \in \Hy \medspace | \medspace z \in D_0 \}.$$
\end{notation}

\begin{definition}
We define the \emph{Bianchi fundamental polyhedron} as 
$$D := D_\infty \cap B.$$
\end{definition}

We can check that the computed polyhedron is indeed a fundamental domain for~$\Gamma$
 using the following observation of Poincar\'e~\cite{Poincare}:
 After a cell subdivision which makes the cell stabilizers fix the cells point-wise,
 the 2-cells (``faces'') of the fundamental polyhedron appear in pairs $(\sigma, \gamma \cdot \sigma)$ with $\gamma \in \Gamma$
 --- so for every orbit of faces, we have exactly two representatives ---
 such that with the orientation for which the lower side of the face $\sigma$ lies on the polyhedron,
 the upper side of~$\gamma \cdot \sigma$ lies on the polyhedron.
We induce a cell structure on~$\Hy$ by the images under~$\Gamma$ of the faces,
 edges and vertices of the Bianchi fundamental polyhedron.

\subsection{The Fl\"oge cellular complex}

In order to obtain a cell complex with compact quotient space, we proceed in the following way due to Fl\"oge \cite{Floege}.
The boundary of $\Hy$ is the Riemann sphere~$\partial \Hy$, which, as a topological space, is made up of the complex plane $\C$ compactified with the cusp $\infty$.
The totally geodesic surfaces in $\Hy$ are the Euclidean vertical planes (we define \emph{vertical} as orthogonal to the complex plane) and the Euclidean hemispheres centred on the complex plane.
The action of the Bianchi groups extends continuously to the boundary \mbox{$\partial \Hy$}. 
The cellular closure of the refined cell complex in $\Hy \cup \partial \Hy $ consists of \mbox{$ \Hy$} and 
\mbox{$\left(\rationals (\sqrt{-m}) \cup \{\infty\}\right) \subset \left(\C \cup \{\infty\}\right) \cong \partial \Hy$.} The $\mathrm{SL_2}(\ringO_{-m})$--orbit of a cusp $\frac{\lambda}{\mu}$ in $\left(\rationals (\sqrt{-m}) \cup \{\infty\}\right)$ corresponds to the ideal class $[(\lambda, \mu)]$ of $\ringO_{-m}$. It is well-known that this does not depend on the choice of the representative $\frac{\lambda}{\mu}$.
We extend the refined cell complex to a cell complex $\widetilde{X}$ by joining to it, in the case that $\ringO_{-m}$ is not a principal ideal domain, the $\mathrm{SL_2}(\ringO_{-m})$--orbits  of the cusps $\frac{\lambda}{\mu}$ for which the ideal $(\lambda, \mu)$ is not principal.
We call the latter cusps the \emph{singular} cusps. 
At the singular cusps, we equip $\widetilde{X}$ with the ``horoball topology'' described in~\cite{Floege}. This simply means that the set of cusps, which is discrete in \mbox{$\partial \Hy$}, is located at the hyperbolic extremities of $\widetilde{X}$ : No neighbourhood of a cusp, except the whole $\widetilde{X}$, contains any other cusp.

We retract $\widetilde{X}$ in the following, $\mathrm{SL_2}(\ringO_{-m})$--equivariant, way.
On the Bianchi fundamental polyhedron, the retraction is given by the vertical projection (away from the cusp $\infty$) onto its facets which are closed in $\Hy \cup \partial \Hy$. The latter are the facets which do not touch the cusp $\infty$, and are the bottom facets with respect to our vertical direction. The retraction is continued on $\Hy$ by the group action. It is proven in \cite{FloegePhD} that this retraction is continuous.
We call the retract of $\widetilde{X}$ the \emph{Fl\"oge cellular complex} and denote it by $X$.
So in the principal ideal domain cases, $X$ is a retract of the refined cell complex, obtained by contracting the Bianchi fundamental polyhedron onto its cells which do not touch the boundary of $\Hy$.  
In \cite{RahmFuchs}, it is checked that the Fl\"oge cellular complex is contractible.
Further details about the Fl\"oge cellular complex and homological computations with it are described in \cite{higher-torsion}.

\subsection{The spectral sequence} \label{spectral}
Let $X$ be our Fl\"oge complex constructed as above.
Next we will consider the spectral sequence associated to the double complex $\textrm{Hom}_{\mathbb Z\Gamma}(\Theta_*,C^*_{\mathbb Z}(X,M))$,
where $\Theta_*$ is the standard resolution of $\mathbb Z$ over $\mathbb Z\Gamma$ and
$C^*(X,M)$ is the cellular co-chain complex of $X$ with $\mathbb Z\Gamma$-module
coefficients $M$. We can (see \cite{Brown}, p. 164) derive the first-quadrant spectral sequence
$$\begin{displaystyle} E^{p,q}_1(M)=  \bigoplus_{\sigma \in \Sigma_p} H^q(\Gamma_{\sigma}; \thinspace  M)\Longrightarrow H^{p+q}(\Gamma; \thinspace  M) \end{displaystyle}$$
where $\Sigma_p$ denotes the $\Gamma$-conjugacy classes of $p$-cells of $X$. Observe that $\Gamma_{\sigma}$ will be a finite group whose order
 is divisible only by 2 and/or 3 unless $\sigma$ is the class of a singular cusp, in which case $\Gamma_{\sigma}$ is a free Abelian group 
on two unipotent generators. 

Assume that $M$ admits an additional module structure over a ring where~$6$ is invertible (in fact we are interested in the case where $M$ is a complex vector space). Then the finite ones among the higher cohomology groups of the $\Gamma_{\sigma}$ vanish. 
Thus, when there are no singular cusps (equivalently, when the class number of $\ringO$ is one), the spectral sequence concentrates on the row $q=0$ and stabilizes on the $E^2$-page. Otherwise, the spectral sequence concentrates on the rows $q=0,1,2$ and stabilizes at the $E^3$-page. 

As we shall see below, the dimension of the module $H^2(\Gamma; \thinspace M)$, which we want to determine, is the same as the dimension of
$$ E_2^{2,0} \simeq E_1^{2,0} / \textrm{Im}(d_1^{1,0}),$$
where the differential $d_1^{1,0}$ is between 
$$\begin{displaystyle} E_1^{1,0} \simeq \bigoplus_{\sigma \in \Sigma_1} M^{\Gamma_{\sigma}} \longrightarrow \bigoplus_{\sigma \in \Sigma_2} M \simeq E_1^{2,0}. \end{displaystyle} $$
The abutment of the spectral sequence gives us
$$ H^2(\Gamma; \thinspace M) \simeq E_3^{2,0} \oplus E_3^{0,2}.$$
Here $E_3^{0,2} \simeq \bigoplus_s H^2(\Gamma_s; \thinspace M)$ where the summation is over $\Gamma$-classes of singular cusps $s$. 

Moreover, 
$E_3^{2,0} = E_2^{2,0} / \textrm{Im}(d_2^{0,1})$ where the differential $d_2^{0,1}$ is between 
$$\begin{displaystyle} \bigoplus_{s \thinspace {\rm singular}} H^1(\Gamma_s; \thinspace M) \longrightarrow E_2^{2,0} .\end{displaystyle} $$

We determine the rank of this differential as follows.

\begin{theorem}[Th\'eor\`eme 8 \cite{Serre}]
Suppose that the coefficient module $M$ is equipped with a non-degenerate $\Gamma$-invariant  
$\C$-bilinear form.
Then the rank of the map from
$H^1(\Gamma; \thinspace  M)$ to the disjoint sum of the $H^1(\Gamma_s; \thinspace  M)$,
induced by restriction from
$H^1(\Gamma; \thinspace  M)$ to $H^1(\Gamma_s; \thinspace  M)$,
equals half of the rank of the disjoint sum of the $H^1(\Gamma_s; \thinspace  M)$.
\end{theorem}
The local topology of this map is studied in~\cite{RahmNoteAuxCRAS}.
The image of this restriction-induced map can be identified with the image of the epimorphism in the short exact sequence
of the spectral sequence's d\'{e}vissage,
$$ 0 \to E_2^{1,0} \longrightarrow H^1(\Gamma; \thinspace M) \longrightarrow {\rm ker} d_2^{0,1} \to 0.$$
Let us assume from now on that $M=E_n$ for some $n$. As we have seen in the proof of Proposition \ref{infinity}, there is a perfect pairing on $M$, which is a non-degenerate $\Gamma$-invariant $\C$-bilinear form. So the theorem of Serre applies, and we obtain the following corollary.
Note for this purpose that the proof of Proposition \ref{infinity} shows that 
$$\dim H^0(\Gamma_s; \thinspace M)= \dim H^2(\Gamma_s; \thinspace M)=1.$$ 
When the cross-section of the cusp $s$ is a torus, we have  
$$\dim H^1(\Gamma_s; \thinspace M)=2\cdot \dim H^2(\Gamma_s; \thinspace M)=2.$$ 
In the cases when $K$ is $\Q(i)$ or $\Q(\sqrt{-3})$, we have 
$$\dim H^1(\Gamma_s; \thinspace M)=0.$$

\begin{corollary} The rank of the differential 
$$\begin{displaystyle} d_2^{0,1}: \bigoplus_{s \thinspace {\rm singular}} H^1(\Gamma_s; \thinspace M) \longrightarrow  E_2^{2,0} \end{displaystyle} $$
is the number of non-trivial ideal classes.
\end{corollary}

\begin{remark} \label{equality of dimensions}
The above discussion implies that
 $$H^2(\Gamma; \thinspace M) \simeq \left ( \bigoplus_{s \thinspace {\rm singular}} H^2(\Gamma_s; \thinspace M) \right ) \oplus \left ( E_2^{2,0} / \textrm{Im}(d_2^{0,1})\right ), $$
and the dimension of $H^2(\Gamma; \thinspace M)$ is the same as that of $E_2^{2,0}$.
\end{remark}

\subsection{The procedure of the computations} 
We compute the representatives of faces in $E_1^{2,0}$ and the differential $d_1^{1,0}$
 of our equivariant spectral sequence with trivial integer coefficients with the program \textit{Bianchi.gp} \cite{BianchiGP}.
 The second author has implemented a MAGMA script that uses the cell stabilizers and identifications obtained with \textit{Bianchi.gp} 
 to compute the action on the coefficient module $M$ that we are interested in.
 We then deduce the term  $E_1^{2,0}$ and the differential $d_1^{1,0}$ with respect to our coefficients.
 The quotient  $$E_2^{2,0} \simeq E_1^{2,0} / \textrm{Im}(d_1^{1,0})$$
now admits the dimension of $H^2(\Gamma; \thinspace M)$
by Remark~\ref{equality of dimensions}.

 As linear algebra over number fields is more expensive 
compared to working over finite fields, we employ the following shortcut. Recall that by the universal coefficients theorem, the dimension of
$H^2(\Gamma; \thinspace M(\F_p))$ (``the mod $p$ dimension")
 is greater than or equal to the dimension of $H^2(\Gamma; \thinspace M(\C))$
 (``the complex dimension''). We start with computing the mod $p$-dimensions for primes $p \leq 200$.
 If we find for a particular $p$ for which the mod $p$
 dimension is equal to the lower bound of Finis-Grunewald-Tirao then we infer that the complex dimension is equal to the mod $p$
 dimension. Note that by Prop. 3.2 (d) of \cite{sengun}, this implies that $H^2(\Gamma; \thinspace M(\mathcal{O}))$
 has no $p$-torsion.
 If this is not the case for the primes in our range, then we compute the complex dimension directly by computing
 $H^2(\Gamma; \thinspace M(K))$. 

\subsection{Execution of the computations}
We applied the above described computations to a database of cell complexes for $186$ Bianchi groups,
 which has been established on the computing clusters of the Weizmann Institute of Science,
 using over fifty processor-months.
This database includes all the cases of ideal class numbers 3 and 5, most of the cases of ideal class number 4 
and {\em all} cases with the absolute value of the discriminant less than 500. 
Almost all of our dimension computations were carried out using the nodes of the computer clusters at the Universities of Duisburg-Essen and Luxembourg.
\bibliographystyle{amsalpha}
\bibliography{literature_2}

\end{document}